\documentclass[10pt,epsfig,amsfonts]{amsart}
\usepackage{stmaryrd}
\usepackage{xcolor}
\usepackage{epsfig}
\usepackage{enumerate}
\usepackage{amsmath}
\usepackage{amssymb}
\usepackage{amscd}
\usepackage{graphicx}
\usepackage{lineno}
\usepackage{pstricks}
\usepackage{tocvsec2}
\usepackage{mathtools}
\allowdisplaybreaks[4]
\RequirePackage[autostyle=true,german=quotes]{csquotes} 

\usepackage [backref, colorlinks, citecolor=red, anchorcolor=black,
  linkcolor=blue]{hyperref}

\usepackage{amsfonts}
\usepackage[top=55mm, bottom=55mm, left=40mm, right=40mm]{geometry}
\usepackage{verbatim}
\usepackage{graphicx, amssymb, amsmath,amsthm}
\allowdisplaybreaks[4]
\usepackage{appendix}
\usepackage{mathrsfs}
\usepackage{cite}

\makeatletter
\@namedef{subjclassname@2020}{%
	\textup{2020} Mathematics Subject Classification}
\makeatother

        \def \F{\mathcal{F}}

\def\N{\mathbb{N}}

\def\diam{\mathrm{diam}}
\def\diam{\mathrm{Diam}}\def\de{\mathcal{DE}^m}\def\xg{(X,G)}
\usepackage{xcolor}

\newcommand{\fs}{\mathcal{FS}^m}

\usepackage{xcolor}
\theoremstyle{plain}

\newtheorem{thm}{Theorem}[section]
\newtheorem{lem}[thm]{Lemma}

\newtheorem{prop}[thm]{Proposition}

\theoremstyle{definition}
\newtheorem{defn}[thm]{Definition}
\newtheorem{rem}[thm]{Remark}

\numberwithin{equation}{section}
\usepackage{epsfig}

\numberwithin{equation}{section}

\renewcommand*{\backref}[1]{}
\renewcommand*{\backrefalt}[4]{\quad \tiny
	\ifcase #1 (\textbf{NOT CITED.})%
	\or    (Cited on Section~#2.)%
	\else   (Cited on Section~#2.)%
	\fi}

\makeatletter

\def\MRbibitem{\@ifnextchar[\my@lbibitem\my@bibitem}

\def\mybiblabel#1#2{\@biblabel{{\hyperref{http://www.ams.org/mathscinet-getitem?mr=#1}{}{}{#2}}}}

\def\myhyperanchor#1{\Hy@raisedlink{\hyper@anchorstart{cite.#1}\hyper@anchorend}}

\def\my@lbibitem[#1]#2#3#4\par{%
	\item[\mybiblabel{#2}{#1}\myhyperanchor{#3}\hfill]#4%
	\@ifundefined{ifbackrefparscan}{}{\BR@backref{#3}}%
	\if@filesw{\let\protect\noexpand\immediate
		\write\@auxout{\string\bibcite{#3}{#1}}}\fi\ignorespaces%
}

\def\my@bibitem#1#2#3\par{%
	\refstepcounter\@listctr
	\item[\mybiblabel{#1}{\the\value\@listctr}\myhyperanchor{#2}\hfill]#3%
	\@ifundefined{ifbackrefparscan}{}{\BR@backref{#2}}%
	\if@filesw\immediate\write\@auxout
	{\string\bibcite{#2}{\the\value\@listctr}}\fi\ignorespaces%
}

\makeatother

\usepackage{xspace} 
\newcommand{\tds}{tds\xspace} 
\newcommand{\mef}{MEF\xspace} 
\newcommand{\eq}{\mathrm{eq}} 
\newcommand{\Folner}{F\o lner\xspace} 
\newcommand{\cF}{\mathcal{F}} 
\newcommand{\cQ}{\mathcal{Q}} 
\DeclareMathOperator{\lD}{\underline{{D}}} 
\DeclareMathOperator{\uD}{\overline{\mathrm{D}}} 
\DeclareMathOperator{\lBD}{\underline{\mathrm{BD}}} 
\DeclareMathOperator{\uBD}{\overline{\mathrm{BD}}} 
\DeclareMathOperator{\Id}{Id}
\newcommand{\ie}{i.\,e.\xspace} 

\usepackage{tikz-cd}

\title[Multivariate diam mean equicontinuity and sensitivity]{A note on
  multivariate diam mean equicontinuity and frequent stability}

\author{Lino Haupt}
\address[Lino Haupt]
{Faculty of Mathematics
  and Computer Science, Friedrich Schiller University Jena, Germany}
\email{linojossfidel.haupt@uni-jena.de}

\author{Tobias Jäger} \address[Tobias Jäger] {Faculty of Mathematics
  and Computer Science, Friedrich Schiller University Jena, Germany}
\email{tobias.jaeger@uni-jena.de}

\author{Chunlin Liu}
\address[Chunlin Liu 1]
{School of Mathematical Sciences, Dalian University of Technology, Dalian, 116024, P.R. China.}

\address[Chunlin Liu 2]{Institute of Mathematics, Polish Academy of Sciences, ul. Śniadeckich 8, 00-656 Warszawa, Poland}
\email{chunlinliu@mail.ustc.edu.cn}

\begin{document}
	 \maketitle
	\begin{abstract}
	Let $(X,G)$ be a topological dynamical system, given by the action of a
        is a countable discrete infinite group on a compact metric space $X$. We
        prove that if $(X,G)$ is minimal, then it is either diam-mean
        $m$-equicontinuious or diam-mean $m$-sensitive. Similarly, $(X,G)$ is
        either frequently $m$-stable or strongly $m$-spreading. Further, when
        $G$ is abelian (or, more generally, virtually nilpotent), then the following
        statements are equivalent:
		\begin{enumerate}
			\item $(X,G)$ is a regular $m$-to-one extension of its
                          maximal equicontinuous factor;
			\item $(X,G)$ is diam-mean $(m+1)$-equicontinuious, and
                          not diam mean $m$-equicontinuious;
			\item $(X,G)$ is not diam-mean $(m+1)$-sensitive, but
                          diam mean $m$-sensitive;
			\item $(X,G)$ has an essential weakly mean sensitive
                          $m$-tuple but no essential weakly mean sensitive
                          $(m+1)$-tuple.
		\end{enumerate}
		This provides a {\em \enquote*{local}} characterisation of $m$-regularity and
        mean $m$-sensitivity vial weakly mean sensitive tuples. The same result
        holds when $G$ is amenable and $(X,G)$ satisfies the local Bronstein
        condition.
	\end{abstract}
	\parskip 1pt

\section{Introduction}

Notions of sensitivity and their relation to equicontinuity properties are a
classical subject of study in the field of topological dynamics. A seminal
result in this context is the Auslander-Yorke dichotomy
\cite{AuslanderYorke1980IntervalMaps}, which states that a minimal topological
dynamical system is either equicontinuous or has sensitive dependence on initial
conditions.  Mean versions of sensitivity and equicontinuity were studied in
\cite{LiTuYe2015meanequicontinuous,downarowicz2016isomorphic}, where it is shown
that mean equicontinuous minimal systems can be characterised by the structural
property that the factor map to their maximal equicontinuous factor (\mef) is a
measure-theoretic isomorphism. Dynamics of this type include regular Toeplitz
flows \cite{Williams1984ToeplitzFlows,Downarowicz2005ToeplitzFlows}, regular
model sets
\cite{Schlottmann1999GeneralizedModelSets,BaakeLenzMoody2007Characterization},
Sturmian subshifts or primitive substitutions with discrete spectrum
\cite{BaakeGrimm2013AperiodicOrder}. They are of particular interest in the
context of aperiodic order, where mean equicontinuous dynamics correspond to
strong forms of long-range order in mathematical models of quasicrystals
\cite{BaakeGrimm2013AperiodicOrder}. An Auslander-Yorke type dichotomy for mean
equicontinous systems has been established in
\cite{LiTuYe2015meanequicontinuous}.  It further turned out that some
structurally defined subclasses of mean equicontinuous minimal system can also
be described by equivalent dynamical characterisations. In particular, a minimal
extension of an equicontinuous system is regular, \ie the factor map is almost
surely injective, if and only if it is diam mean equicontinuous, and it is
almost one-to-one, \ie the factor map is injective on a residual set,
if and only if it is frequently stable
\cite{GarciaJaegerYe2021DiamMeanEquicontinuity}.

More recently, multivariate forms of mean equicontinuity and sensitivity have
attracted considerable attention
\cite{ShaoYeZhang2008NSensitivity,LiYeYu2022Equiinthemean,LiYu2021MeanSensitiveTuples,LiuYin2023MeanSensitiveTuplesOfGroupActions,HuangLianShaoYe2021MinimalSystemWithFinitelyManyMeasures}. While
mean equicontinuity corresponds -- as mentioned above -- to a one-to-one
relation between the system and its \mef, the multivariate analogue is
characteristic of finite-to-one extensions of equicontinuous systems
\cite{BreitenbucherHauptJaeger2024FiniteTopomorphicExtensions}. These again
include a variety of paradigmatic examples, such as the Thue-Morse substitution
and its generalisations \cite{BaakeGrimm2013AperiodicOrder}, certain irregular
Toeplitz flows constructred by Williams \cite{Williams1984ToeplitzFlows}, Iwanik
and Lacroix \cite{IwanikLacroix1994NonRegularToeplitz}, and others, or certain
irregular model sets
\cite{FuhrmannGlasnerJaegerOertel2021TameImpliesRegular}. Again, an
Auslander-Yorke dichotomy holds for multivariate mean equicontinuous minimal
systems, as shown in
\cite{BreitenbucherHauptJaeger2024FiniteTopomorphicExtensions}.

The purpose of this note is to complement and extend this ongoing research by
establishing Auslander-Yorke type dichotomies for the multivariate analogues of
diam mean equicontinuity and frequent stability. Moreover, we provide a
localised characterisation of multivariate mean sensivity in terms of weakly
mean sensitive tuples. In order to state our main results, we need to provide
some definitions.
\medskip

We call a pair $(X,G)$ a {\em topological dynamical system} (\tds) if $X$ is a
compact metric space and $G$ is a countably infinite discrete amenable group that acts on
$X$ by homeomorphisms.  With slight abuse of notation we write $g:X\to
X,\ x\mapsto gx$ for $g\in G$.  A \tds is called {\em transitive} if there
exists a point with dense orbit and {\em minimal} if all orbits are dense.
A point \( x \in X \) is called an \emph{almost periodic point} (or a \emph{minimal point}) if the closure of its orbit is a minimal system.

 It is
well-known that any \tds has a {\em maximal equicontinuous factor (\mef)}
\cite{A,Downarowicz2005ToeplitzFlows}, which we denote by $(X_{\eq},G)$. The
corresponding factor map will be denoted by $\pi_{\eq}:X\to X_{\eq}$. When
$(X,G)$ is transitive, its \mef is strictly ergodic with a unique invariant
Borel probability measure $\mu_{\eq}$. We call $(X,G)$ {\em $m$-regular} if
$\mu_{\eq}\left(\left\{y\in X_{\eq}\mid \sharp
\pi_{\eq}^{-1}(y)=m\right\}\right)=1$ and {\em almost $m$:1} if $\left\{y\in
X_{\eq}\mid \sharp \pi_{\eq}^{-1}(y)=m\right\}$ is residual.  Given $m\in\N$, we
define the {\em $m$-separation} of $m$ points $x_1,\ldots, x_m\in X$ as
\[
D_m(x_1,\ldots x_m)\ = \ \min_{1\leq i < j \leq m} d(x_i,x_j)\ ,
\]
where $d$ denotes the metric on $X$. Given a subset $A\subseteq X$, the
$m$-diameter of $A$ is given by
\[
\diam_m(A) \ = \ \sup\left\{\left. D_m(x_1,\ldots,x_m)\ \right|
\ (x_1,\ldots,x_m)\in A^m\right\} \ .
\]
A \tds $(X,G)$ is said to be {\em diam-mean $m$-equicontinuous} if for any
$\varepsilon>0$ there exists $\delta>0$ such that for all sets $U\subset X$ with
$\diam_m(U)<\delta$ we have
\begin{equation}\label{e.m_diam_mean_equi}
 \sup_{\cF}\limsup_{n\to\infty}\frac{1}{|\cF_n|}\sum_{g\in
   \cF_n}\operatorname{Diam}_m(g(U)) < \ \varepsilon \ .
\end{equation}
Here and in all of the following $\sup_\F$ is the supremum taken over all
\Folner sequences $\cF$ of $G$.  Furthermore, we understand $\cF_n$ to be the
$n$-th element of $\cF = (\cF_n)_{n \in \N}$.  The significance of this
dynamical property of diam-mean $m$-equicontinuity lies in the fact that --- at
least in the case of minimal abelian group actions --- a \tds is $m$-regular if
and only if it is diam mean $(m+1)$-equicontinuous, but not diam mean
$m$-equicontinuous \cite{Lino2025}.\footnote {
		\label{footnote:atmostvsexact}
		There is a slight mismatch between the notation here and in
                \cite{Lino2025} that might cause confusion.  In \cite{Lino2025}
                an \enquote{almost surely $m$-to-1} extension is a factor map
                where almost all points have \emph{at most} $m$ preimages.  Here
                an \enquote{$m$-regular} extension is a factor map where almost
                all points have \emph{exactly} $m$ preimages.  } We will obtain
the same result under more general conditions on the group --- including the
nilponent case --- as a byproduct in Section~\ref{WeaklyMeanSensitiveTuples}
(Theorem~\ref{thm:m to 1<=> tuple}).



A point $x\in X$ is said to be a {\em diam-mean $m$-equicontinuity point} if for
any $\varepsilon>0$, there exists $\delta>0$ such that (\ref{e.m_diam_mean_equi})
holds for all open neighbourhoods $U$ of $x$ with $\diam_m(U)<\delta$. It is a
direct consequence of the definitions and compactness that $(X,G)$ is diam-mean
equicontinuous if and only if all $x\in X$ are diam-mean equicontinuity
points. $\xg$ is called {\em almost diam-mean $m$-equicontinuous} if the set of diam-mean
$m$-equicontinuous points is a residual subset of $X$.  Conversely, a \tds
$(X,G)$ is said to be {\em diam-mean $m$-sensitive} if there exists $\varepsilon>0$
such that for any nonempty open subset $U\subset X$
\begin{equation}\label{e.diam_mean_sensitivity}
\sup_{\cF}\limsup_{n\to\infty}\frac{1}{|\cF_n|}\sum_{g\in
  \cF_n}\operatorname{Diam}_m(g(U)) \ \geq\ \varepsilon \ .
\end{equation}
A point $x\in X$ is said to be {\em diam-mean $m$-sensitive} if there exists
$\varepsilon>0$ such that (\ref{e.diam_mean_sensitivity}) holds for any open
neighbourhood $U$ of $x$.  Our first main result is an Auslander-Yorke type
dichotomy for these notions.

\begin{thm} \label{t.Auslander_Yorke_dichotomy_diam_mean}
  Suppose that $(X,G)$ is a tds and $m\in\N$.
  \begin{itemize}
	  \item[(i)]
		  If $(X,G)$ is transitive, then it is either
		  almost diam-mean $m$-equicontinuous or diam-mean $m$-sensitive.
	  \item[(ii)]
		  If $(X,G)$ is minimal, then it is either
		  diam-mean $m$-equicontinuous or diam-mean $m$-sensitive.
  \end{itemize}
\end{thm}

A weaker notion than that of diam mean equicontinuity is that of frequent stability.
We call $(X,G)$ {\em frequently $m$-stable} if for all $\varepsilon>0$ there
exists $\delta>0$ such that
\begin{equation}\label{e.frequent_stability}
	\uBD\left(\left\{ g\in G\mid \diam_m(gA) < \varepsilon\right\}\right) \ > \ 0
\end{equation}
holds for all sets $A\subseteq X$ with $\diam(A)<\delta$, where $\uBD(S)$
denotes the upper Banach density of a set $S\subseteq G$. Conversely, we call
$(X,G)$ {\em strongly $m$-spreading} if there exists $\varepsilon>0$ such that
\begin{equation}\label{e.strong_spreading}
\uBD\left(\left\{ g\in G\mid \diam_m(gU) < \varepsilon\right\}\right) \ = \ 0
\end{equation}
holds for all open sets $U\subseteq X$. Similar as for diam mean equicontinuity,
frequent stability is related to a structural property of \tds: in the case of
minimal abelian group actions, a \tds is an almost $m$:1-extension of its \mef if
and only if it is frequently $(m+1)$-stable, but not frequently
$m$-stable \cite{Lino2025}.\footnote
	{
		\label{footnote:atmostvsexact2}
		As explained in Footnote \ref{footnote:atmostvsexact}, there is a mismatch between the notation here and in \cite{Lino2025}:
		In \cite{Lino2025} an \enquote{almost $m$-to-1} extension is a factor map
		where a residual set of points have \emph{at most} $m$ preimages.
		In our notation here, an \enquote{almost $m$:1} extension is a factor map
		where a residual set of points have \emph{exactly} $m$ preimages.
	}

We again consider pointwise verions of these properties:
An element $x\in X$ is called a {\em frequent $m$-stability point}
if for all $\varepsilon>0$ there exists
$\delta>0$ such that (\ref{e.frequent_stability}) holds for all $A\subseteq X$
with $x\in A$ and $\diam(A)<\delta$.
We say $x$ is a {\em strong $m$-spreading point}
if there exists $\varepsilon>0$ such that (\ref{e.strong_spreading}) holds
for all open neighbourhoods of $x$. 
If $x$ is strongly $m$-spreading with respect to a given $\varepsilon > 0$ we will call $x$
a \emph{strong $\varepsilon$-$m$-spreading point}.
Finally, we call $(X,G)$ {\em almost
  frequently $m$-stable} if the set of frequently $m$-stable points is a
residual subset of $X$. As before, we obtain an Auslander-Yorke type dichotomy.

\begin{thm} \label{t.Auslander_Yorke_dichotomy_frequent_stability}
  Suppose that $(X,G)$ is a tds and $m\in\N$. Then the
  following hold.
  \begin{itemize}
  \item[(i)]
	  If $(X,G)$ is transitive, then it is either
	  almost frequently $m$-stable or strongly $m$-spreading.
  \item[(ii)] 
	  If $(X,G)$ is minimal, then it is either
	  frequently $m$-stable or strongly $m$-spreading.
  \end{itemize}
\end{thm}

 Our second main objective is to provide {\em \enquote*{local}} characterisations of the
 above sensitivity properties. In spirit, this is inspired by the
 characterisation of positive entropy via entropy tuples introduced in
 \cite{Blanchard1993DisjointnessTheoremForTopologicalEntropy,BlanchardHostMaassMartinez1995EntropyPairsForMeasure,BlanchardGlasnerHost1997VariationalPrincipleAndEntropyPairs}. Sensitive
 tuples were first defined in \cite{ShaoYeZhang2008NSensitivity} to study
 sensitivity in topological dynamics. In the context of mean sensitivity, the
 notion of mean sensitive tuples were studied
 \cite{LiYeYu2022Equiinthemean}. Here, we provide analogue results for diam mean
 equicontinuity. 

 We say a $K$-tuple $ (x_k)_{k=1}^K\in X^K$ is a {\em weakly mean-sensitive
   $K$-tuple} if for any open neighborhood $U_k$ of $x_k$ there exists $\eta>0$
 such that for any non-empty open subset $U$ of $X$,
 \[
 \uBD\left(\left\{g\in G: (gU)\cap U_k\neq\emptyset\text{ for }1\le k\le
 K\right\}\right)\ > \ \eta \ .
 \]
 A tuple $x\in X^K$ is called {\em essential} if $x_i\neq x_j$ for all $i\neq
 j\in\{1,\ldots,K\}$.

 \begin{thm}
   Suppose that $(X,G)$ is a transitive \tds and $m\in\N$. Then $(X,G)$ is diam
   mean $m$-sensitive if and only if there exists an essential weakly
   $m$-sensitive tuple.
 \end{thm}

 A similar localised characterisation of strong $m$-spreading is more intricate,
 since it is difficult to ensure the required spreading property along a subset
 $S\subseteq G$ of full density with a single tuple. However, we provide an
 equivalent condition for strong $m$-spreading in terms of finite families of
 weakly $m$-sensitive tuples in Section~\ref{Strong_m_spreading} (see
 Prop.~\ref{p.strong_spreading_via_tuples}).
 \medskip

 \noindent
{\bf Structure of the article.}\quad In Section~\ref{Prelim}, we provide the
necessary background and preliminaries. The Auslander-Yorke dichotomies for diam
mean $m$-sensitivity and $m$-equicontinuity as well as for frequent
$m$-stability and strong $m$-spreading are proven in
Section~\ref{AuslanderYorke} and Section \ref{sec:4}. Section~\ref{WeaklyMeanSensitiveTuples} then
contains the characterisation of these notions in terms of weakly mean
$m$-sensitive tuples.

\medskip

\noindent
{\bf Acknowledgements.} LH and TJ were supported by grant 552775134 of the
German Research Council.

\section{Preliminaries} \label{Prelim}
\subsection{Notation}
Let $A$ and $B$ be sets.
By $|A|$ we denote the cardinality of $A$.
The symmetric difference of $A$ and $B$ shall be written as $A \triangle B$.
If $A, B \subseteq G$ for some group $G$, then we let
\[ A B := \left\{ a b \left\vert a \in A, b \in B \right. \right\} \,. \]

If $(X,d)$ is a metric space, we will denote the open ball around $x \in X$ with radius $\delta$ by
\[ B(x,\delta) := \{ x' \in X \mid d(x,x') < \delta \} \,.\]
Note that we suppress the metric and space notationally.
In particular this will mean, that $B(x,\delta)$ might be a ball in the original \tds and $B(\pi(x),\varepsilon)$ a ball in the \mef.

\subsection{\Folner Sequences and Densities}
Let $G$ be an countably infinite and discrete group.
\begin{defn}[\Folner Sequence]
	A sequence $\mathcal{F} = (\cF_n)_{n\in\N}$ of finite subsets $\cF_n \subseteq G$ is called \emph{\Folner} if
	for any finite $K \subseteq G$ we have
	\[ \lim_{n\to\infty} \frac{| \cF_n \triangle K \cF_n |}{|\cF_n|} = 0 \,. \]
\end{defn}
\begin{defn}[Amenability]
	We call $G$ \emph{amenable} if at least one \Folner sequence exists.
\end{defn}

From now on, we assume $G$ to be amenable.  Using \Folner sequences one can
define a variety of invariant means on the group $G$. The most simple of these
are densities.

\begin{defn}[Densities]
	Let $A \subseteq G$ be arbitrary and fix a \Folner sequence $\cF$.
	We define the \emph{upper density} of $A$ as
	\[ \uD_\F(A) := \limsup_{n\to\infty} \frac{|A  \cap \cF_n|}{|\cF_n|} \,.\]

	Similarly, we define the \emph{lower density} of $A$ as
	\[ \lD_\F(A) := \liminf_{n\to\infty} \frac{|A  \cap \cF_n|}{|\cF_n|} \,.\]
\end{defn}

Often we want to talk about all \Folner sequences at once.
We therefore define
\begin{defn}[Banach Densities]
	Let $A \subseteq G$ be arbitrary.
	We define the \emph{upper Banach density} of $A$ as
	\[ \uBD(A) := \sup_{\cF} \limsup_{n\to\infty} \frac{|A  \cap \cF_n|}{|\cF_n|} \,,\]
	where $\sup_{\cF}$ --- as always --- is the supremum taken over all \Folner sequences.

	Similarly, we define the \emph{lower Banach density} of $A$ as
	\[ \lBD(A) := \inf_{\cF} \liminf_{n\to\infty} \frac{|A  \cap \cF_n|}{|\cF_n|} \,, \]
	where $\inf_{\cF}$ --- as always --- is the infimum taken over all \Folner sequences.
\end{defn}

\subsection{MEF}
Recall that a map $p:X\to Y$ between topological dynamical systems $(X,G)$ and $(Y,G)$
is called a \textit{factor map} if it is \textit{equivariant}, \ie
\[ \forall g \in G: p(g(x)) = g(p(x)) \,. \]
Here --- and in the remainder --- we notationally identify $g \in G$,
with its corresponding homeomorphism on the state space.
If $p:X \to Y$ is a factor map, we write $p:(X,G) \to (Y,G)$.
A factor map $p:(X,G) \to (Y,G)$ is called an \textit{conjugacy} if there is $q: (Y,G) \to (X,G)$
with $p \circ q = \Id_Y, q \circ p = \Id_X$.

\begin{defn}[Equicontinuity]
	A \tds $(Y,G)$ is called \emph{equicontinuous},
	if for any $\varepsilon > 0$ there is $\delta > 0$ such that
	for any $y_1, y_2 \in Y$ with $d_Y(y_1,y_2) < \delta$ we have
	\[ \forall g \in G: d_Y(g(y_1),g(y_2)) < \varepsilon \,, \]
	where $d_Y$ is the metric on $Y$.
\end{defn}
A straightforward calculation shows that 
\[ d'_Y(y_1,y_2) := \sup_{g \in G} d_Y(g(y_1),g(y_2)) \]
is an invariant metric that is topologically equivalent to the original metric $d_Y$.
Thus we can --- and will --- always assume that metrics on equicontinuous \tds are invariant.

Now let $(X,G)$ be an arbitrary \tds.
Clearly, the trivial one-point \tds $(\{ \ast \}, \Id)$ is an equicontinuous factor of $(X,G)$.
However we are interested in an equicontinuous factor which is \textit{maximal}.
\begin{defn}[Maximal Equicontinuous Factor]
	\label{def:mef}
	A factor $(Y,G)$ --- together with the factor map $\pi: (X,G) \to (Y,G)$ ---
	is a \emph{maximal equicontinuous factor} (\mef) of $(X,G)$
	if
	$(Y,G)$ is equicontinuous and
	for any factor map $p: (X,G) \to (Z,G)$ onto an equicontinuous \tds $(Z,G)$
	there is a unique factor map $q: (Y,G) \to (Z,G)$ such that $p = q \circ \pi$.
	\begin{figure}[ht]
		\centering
		\catcode`\"=12
		\begin{tikzcd}
			(X,G) \arrow[rrdd, "p"] \ar[dd, "\pi"]  & & \\
			& & \\
			 (Y,G) \arrow[rr, "\exists!\, q", dotted]                &  & (Z,G)
		\end{tikzcd}
		\caption{Universal Property of \mef}
		\label{fig:universalpropertyMEF}
	\end{figure}
\end{defn}

The existence and uniqueness --- up to conjugacy --- of the \mef is classical \cite[Theorem 1]{EG60}.
\begin{thm}
	\label{thm:existenceuniquenessMEF}
	Any \tds $(X,G)$ has a \mef $(Y,G)$.
	If $(Z,G)$ is another \mef, then there is a conjugacy $h: (Y,G) \to (Z,G)$.
\end{thm}
As the \mef of $(X,G)$ is unique, we may give it a name and call it $(X_{\eq},G)$.
The factor map to the \mef shall be denoted by $\pi_{\eq}: (X,G) \to (X_{\eq},G)$.

\subsection{Regionally Proximal Relation}
The above description of the \mef is abstract.
The proof of Theorem \ref{thm:existenceuniquenessMEF} in \cite[Theorem 1]{EG60} is more constructive
but does not reveal any dynamical characterization of the equivalence relation induced by the \mef.

Such a characterization is given by the \textit{regionally proximal relation}.
\begin{defn}[Regionally Proximal Relation]
	\label{def:regionallyproximal}
	The \emph{regionally proximal relation} $\cQ$ is given by
	$(x,y) \in \cQ$ if and only if there are sequences $(x'_n)_{n\in\N}, (y'_n)_{n\in\N} \in X^\N$ and $(g_n)_{n\in\N} \in G^\N$ such that
	\[ 
		\lim_{n\to\infty} x'_n = x,  \lim_{n\to\infty} y'_n = y \text{ and } 	\lim_{n\to\infty}d(g_n(x'_n),g_n(y'_n)) = 0 \,.
	\]
\end{defn}

The link to the \mef is established by the following result:
\begin{thm}[{\cite[Chapter 9, 3. Theorem]{A}}]
	Let $(X,G)$ be any \tds.
	Let $\cQ^*$ be the smallest equivalence relation containing $\cQ$.
	Then the \tds $(X/{\cQ^*},G)$, together with the factor map
	\[ \pi: (X,G) \longrightarrow (X/{\cQ^*},G), \; x \longmapsto \cQ^*[x] \,,\]
	is the \mef of $(X,G)$.
\end{thm}

Relevant for us are also the multivariate notions.
\begin{defn}[Multivariate Regionally Proximal Relation]
	\label{def:multivariateregionallyproximal}
	Let $m \in \N$ with $m \geqslant 2$ and define the \emph{regionally proximal $m$-relation}, by
	$(x_1,\ldots,x_m) \in \cQ_m$ if and only if there are sequences $(x'_{1,n})_{n\in\N}, \ldots, (x'_{m,n})_{n\in\N} \in X^\N$ and $(g_n)_{n\in\N} \in G^\N$ such that
	\[ 
		\forall i \in \{ 1,\ldots, m \}: \lim_{n\to\infty} x'_{i,n} = x_i, \text{ and } \lim_{n\to\infty} \max_{i\neq j} d(g_n(x'_{i,n}),g_n(y'_{j,n})) = 0 \,.
	\]
\end{defn}

The dynamics inside of the classes of the regionally proximal relation is sensitive:
\begin{lem}
	Assume that $(X,G)$ is minimal.
	Then $(x_1,\ldots,x_m) \in \cQ_m$
	if and only if
	for any non-empty and open $B \subseteq X$
	and for any collection neighbourhoods $U_1,\ldots,U_m$ of the points $x_1,\ldots,x_m$ there is $h \in G$ such that
	$h(B) \cap U_i \neq \emptyset$ for $i \in \{ 1, \ldots m \}$.
\end{lem}
A proof can be found for example in \cite{huang2011measure} and \cite[Lemma 2.7]{Lino2025}.  For a large
class of systems the multivariate regionally proximal relation is exactly given
by the pairwise regionally proximal relation and the regionally proximal
relation is already a (closed) equivalence relation.

\begin{thm}[{\cite[Theorem 8]{AuslanderGroupTheoretic}}]
	\label{thm:auslandersmagic}
	Assume that $(X,G)$ is locally Bronstein (see below).
	For any $m \in \N$ and any $(x_1,\ldots,x_m)\in X^m$ we have
	\[ (x_1,\ldots,x_m) \in \cQ_m \quad \text{ if and only if } \quad (x_1,x_i) \in \cQ \text{ for any } i \in \{1,\ldots,m\} \,. \]
\end{thm}
\begin{rem}
	\label{rem:auslandersmagic}
	This is not the exact statement of \cite[Theorem 8]{AuslanderGroupTheoretic}.
	However combining lemmas from \cite{AuslanderProximalCell} and \cite{AuslanderGroupTheoretic}
	one obtains Theorem \ref{thm:auslandersmagic}.
\end{rem}

Here the class of locally Bronstein systems is defined as follows:
\begin{defn}[Locally Bronstein]
	A system is called \emph{locally Bronstein} if for each pair $(x,y) \in
        X^2$ which is almost periodic in $(X^2,G)$ and satisfies $x \cQ y$ there
        are sequences $(x''_{n})_{n\in\N}, (y''_{n})_{n\in\N} \in X^\N$ with
        $(x''_n,y''_n)$ almost periodic in $(X^2,G)$ and there is a sequence
        $(g_n)_{n\in\N} \in G^\N$ such that
	\[ 
		\forall i \in \{ 1,\ldots, m \}: \lim_{n\to\infty} x''_{i,n} =
                x_i, \text{ and } \lim_{n\to\infty} \max_{i\neq j}
                d(g_n(x''_{i,n}),g_n(y''_{j,n})) = 0 \ . 
	\]
	In other words, the approximating points $(x''_{n})_{n\in\N},
        (y''_{n})_{n\in\N} \in X^\N$ in Definition \ref{def:regionallyproximal}
        can be chosen almost periodic. .
\end{defn}

For many acting groups the local Bronstein property is satisfied automatically.
\begin{lem}
	\label{lem:nilpotentimpliesbronstein}
	If the acting group $G$ is virtually nilpotent --- in particular, abelian --- 
	any minimal system $(X,G)$ is \textit{incontractible}, \ie
	the almost periodic points of $(X^n,G)$ are dense in $X^n$ for any $n \in \N$.
\end{lem}
\begin{rem}
	\begin{enumerate}[a)]
		\item
			A straightforward argument (just using the case $n=2$)
                        shows that any incontractible system is locally
                        Bronstein.
		\item
			The main source of Lemma
                        \ref{lem:nilpotentimpliesbronstein} is (the proof of)
                        \cite[Theorem 7.5]{GlasnerCompressibility}.  However, it
                        is not immediatly clear how to connect the proof and
                        statement of \cite[Theorem 7.5]{GlasnerCompressibility}
                        with Lemma \ref{lem:nilpotentimpliesbronstein}.  The way
                        to bridge that gap is spread across the Literature.  For
                        ease of the reader and precise referencing we sketch
                        that connection may be done:

			In \cite[Definition 4.3]{GlasnerBernoulliDisjointness}
                        the above definition of \textit{incontractible systems}
                        is given.  Minimal incontractible systems are
                        characterized by disjointness to any minimal proximal
                        flow \cite[Proposition
                          4.4]{GlasnerBernoulliDisjointness}.\footnote { For a
                          minimal system \textit{incontractibility} is
                          equivalent to \textit{un-contractibility} --- in the
                          sense of \cite{GlasnerCompressibility}.  This is due
                          to the characterization in \cite[Theorem
                            4.2]{GlasnerCompressibility} which states that for
                          minimal system non-contractibility is agian equivalent
                          to disjointness to any minimal proximal flow.  } In
                        the proof of \cite[Theorem 7.5]{GlasnerCompressibility},
                        we now see that for virtually nilpotent acting groups
                        the minimal proximal systems are all trivial.  Thus the
                        condition of being disjoint to all of them is vacuously
                        satisfied.  Hence all systems with a virtually nilpotent
                        acting group are incontractible.
	\end{enumerate}
\end{rem}

\section{Auslander-Yorke type dichotomies for multivariate diam mean equicontinuity and frequent stability}
\label{AuslanderYorke}
The argument for Auslander-Yorke type dichotomies usually follow a similar
structure: The set of $\varepsilon$-sensitivity points --- for each sensitivity
constant $\varepsilon$ --- is invariant and closed.  Therefore, if a
transitivity point is sensitive --- and thus $\varepsilon$-sensitive for some
$\varepsilon > 0$ --- the whole system is sensitive.

\subsection{Multivariate frequent stability and strong spreading}
\label{FrequentStabilityAndSpreading}

 Let $\xg$ be a \tds.
Following ideas in \cite{LiTuYe2015meanequicontinuous}, we study the set of
frequently $m$-stability points, denoted by $\fs$.
For every $\varepsilon>0$, let
\begin{align}
	\label{eq:defFS}
	\fs_\varepsilon \ := \ \left\{x\in X \ \middle| \ \exists\delta>0:
	\uBD\left(\left\{ g\in G\mid \diam_m(g B(x,\delta)) < \varepsilon\right\}\right) \ > \ 0 
	\right\}.
\end{align}
Note that 
$(\fs_\varepsilon)^c$ is the set of strong $\varepsilon$-$m$-spreading points.
\begin{prop}
	\label{prop:gdeltaandinvariancefreqstab} Let $\xg$ be a \tds, and
	$\varepsilon>0$.  Then $\fs_\varepsilon$ is open and invariant, \ie
	$h(\fs_\varepsilon)=\fs_\varepsilon$ for any $h\in G$.  Consequently,
	the set of strong $\varepsilon$-$m$-spreading points is closed and
	invariant.  Furthermore, as $\fs=\cap_{n=1}^\infty\fs_{1/n}$ we conclude
	that $\fs$ is $G_\delta$.
\end{prop}
\begin{proof}
	For any $x\in \fs_\varepsilon$, we choose $\delta>0$ --- as in the
	definition of $\fs_\varepsilon$ in Equation (\ref{eq:defFS}) --- such
	that \[ \uBD\left(\left\{ g\in G\mid \diam_m(g B(x,\delta))
	< \varepsilon\right\}\right) \ > \ 0 \,.\] Note that for any $y\in
	B(x,\delta/2)$, we have $B(y,\delta/2)\subset B(x,\delta)$. Therefore,
	we obtain \[ \uBD\left(\left\{ g\in G\mid \diam_m\left(g
	B\left(y,\frac{\delta}{2}\right)\right) < \varepsilon\right\}\right) \
	> \ 0 \,.\] Thus $y\in\fs_\varepsilon$, and hence $\fs_\varepsilon$ is
	open.
			
	In order to show the invariance let $h \in G$ be arbitrary and assume
	that $x \in X$ such that $hx\notin\fs_\varepsilon$.  We will show
	$x\notin\fs_\varepsilon$.  As $h: X \to X$ is a homeomorphism
	$hB(x,\delta)$ is open for any $\delta > 0$.  Hence, for any $\delta >
	0$ there exists $\eta_\delta >0$ such that $B(hx,\eta_\delta)\subset
	hB(x,\delta)$.

	By assumption, \[ \limsup_{n\to\infty}\frac{1}{|\cF_n|} |\left\{
	g\in \cF_n \mid \diam_m\left(g B\left(hx,\eta_\delta\right)\right)
	< \varepsilon\right\}| = 0 \] for any \Folner sequence $\cF$ and $\delta
	> 0$.  Therefore, \begin{align*}
	&\limsup_{n\to\infty}\frac{1}{|\cF_n|} |\left\{
	g\in \cF_n \mid \diam_m\left(g B\left(x,\delta \right)\right)
	< \varepsilon\right\}| \\ =&\limsup_{n\to\infty}\frac{1}{|\cF_n|} |\left\{
	g\in \cF_n\cdot
	h^{-1} \mid \diam_m\left((gh)B\left(x,\delta\right)\right)
	< \varepsilon\right\}| \\ \leqslant&\limsup_{n\to\infty}\frac{1}{|\cF_n|} |\left\{
	g\in \cF_n\cdot
	h^{-1} \mid \diam_m\left(gB\left(hx,\eta_\delta\right)\right)
	< \varepsilon\right\}|. \end{align*} Since $\{\cF_nh^{-1}\}_{n=1}^\infty$
	is also a \Folner sequence, the latter is $0$ by assumption.  So
	$x \notin \fs_\varepsilon$.  By contraposition we obtain that
	$x \in \fs_\varepsilon$ implies $hx \in \fs_\varepsilon$ --- \ie the
	desired invariance.  Replacing $x$ by $h(x)$ and $h$ by $h^{-1}$, we
	also obtain the reverse inclusion and the invariance of
	$(\fs_\varepsilon)^c$.
\end{proof}

\begin{prop}[Auslander-Yorke Dichotomy for Frequent Stability]
	\label{prop1}
	\label{prop:auslanderyorkefreqstab}
	Let $\xg$ be a transitive \tds.
	Then $\xg$ is either almost frequently $m$-stable
	--- \ie $\fs$ is dense
	$G_\delta$-set ---
	or $\xg$ is strongly $m$-spreading  --- \ie $\fs = \emptyset$.

	If in addition, $\xg$ is minimal, then either $\xg$ is frequently $m$-stable
	--- \ie $\fs = X$ ---
	or $\xg$ is strongly $m$-spreading
	--- \ie $\fs=\emptyset$.
\end{prop}
\begin{proof}
	Let $\xg$ be transitive and let $x$ be a transitivity point.

	Assume on the one hand that $x \in \fs$.
	Then $\fs$ is a dense and by Proposition \ref{prop:gdeltaandinvariancefreqstab}
	$\fs$ is a $G_\delta$-set.
	Thus $\fs$ is dense $G_\delta$ and hence $\xg$ is
	almost frequently $m$-stable.

	Assume on the other hand that $x \notin \fs$.
	Then there is $\varepsilon > 0$ such that
	$x \in (\fs_\varepsilon)^c$.
	By the invariance of $(\fs_\varepsilon)^c$, we have $Gx \subseteq (\fs_\varepsilon)^c$.
	Finally, by the closedness, we have $X = \overline{Gx} \subseteq (\fs_\varepsilon)^c$.
	Therefore $\xg$ is strongly $m$-spreading.
	It is in particular clear, that the cases are disjoint.

	Let $\xg$ now be minimal.  Then any $x \in X$ is a transitivity point.
	Assume that there is $\varepsilon > 0$ and $y \in X$ such that $y \in
	(\fs_\varepsilon)^c$.  Then again --- as $y$ has dense orbit --- so $X =
	(\fs_\varepsilon)^c$.  So $\fs = \emptyset$.
\end{proof}
\subsection{Multivariate diam-mean equicontinuity and sensitivity}
\label{DiamMeanEquicontinuityAndSensitivity}

Following ideas in \cite{LiTuYe2015meanequicontinuous}, we study the set of
diam-mean $m$-equicontinuous points, denoted by $\mathcal{DE}^m$.
For every $\varepsilon>0$, let
\begin{align}
	\label{eq:defDE} \de_\varepsilon \ := \ \left\{x\in
	X \ \left| \ \exists\delta>0: \ \sup_{\cF}\limsup_{n\to\infty}\frac{1}{|\cF_n|}\sum_{g\in \cF_n}\operatorname{Diam}_m(g(B(x,\delta)))<\varepsilon\right.\right\}\
	.
\end{align}
Then
	$\de=\cap_{n=1}^\infty\de_{1/n}$.
Note that 
$(\de_\varepsilon)^c$ is the set of diam-mean $\varepsilon$-$m$-sensitivity points.
\begin{prop}\label{prop:equi. point}
	\label{prop:gdeltaandinvariancediammean} Let $\xg$ be a \tds, and
	$\varepsilon>0$.  Then $\de_\varepsilon$ is open and invariant, \ie
	$h(\de_\varepsilon)=\de_\varepsilon$ for any $h\in G$.  Consequently,
	the set of diam-mean $\varepsilon$-$m$-sensitivity points is closed and
	invariant.  Furthermore, as $\de=\cap_{n=1}^\infty\de_{1/n}$ we conclude
	that $\de$ is $G_\delta$.
\end{prop}
\begin{proof}
	For any $x\in \de_\varepsilon$, we choose $\delta>0$ --- as in the
	definition of $\de_\varepsilon$ in Equation (\ref{eq:defDE}) --- such
	that \[ \sup_{\cF}\limsup_{n\to\infty}\frac{1}{|\cF_n|}\sum_{g\in\cF_n} \diam_m(g(B(x,\delta)))<\varepsilon\,.  \]
	Note that for any $y\in B(x,\delta/2)$, we have $B(y,\delta/2)\subset
	B(x,\delta)$. Therefore, we
	obtain \[\sup_{\cF}\limsup_{n\to\infty}\frac{1}{|\cF_n|}\sum_{g\in \cF_n}\operatorname{Diam}_m(g(B(y,\delta/2)))\
	<\ \varepsilon\ .\] Thus $y\in \de_\varepsilon$, and hence
	$\de_\varepsilon$ is open.
			
	In order to show the invariance let $h \in G$ be arbitrary and assume
	that $x \in X$ such that $hx\notin\de_\varepsilon$.  We will show
	$x\notin\fs_\varepsilon$.  As $h: X \to X$ is a homeomorphism
	$hB(x,\delta)$ is open for any $\delta > 0$.  Hence, for any $\delta >
	0$ there exists $\eta_\delta >0$ such that $B(hx,\eta_\delta)\subset
	hB(x,\delta)$.

	By assumption there is a \Folner sequence $\cF$ such
	that \[ \limsup_{n\to\infty}\frac{1}{|\cF_n|}\sum_{g\in \cF_n}\diam_m(g(B(hx,\delta))) \ \geq \ \varepsilon\]
	for any $\delta > 0$.  Let $\delta > 0$ be arbitrary,
	then \begin{eqnarray} \nonumber \lefteqn{\limsup_{n\to\infty}\frac{1}{|\cF_n|}\sum_{g\in \cF_n}\operatorname{Diam}_m(g(B(x,\delta))) \
	} \\ & =
	& \nonumber \limsup_{n\to\infty}\frac{1}{|\cF_n|}\sum_{g\in \cF_nh^{-1}}\operatorname{Diam}_m(gh(B(x,\delta))) \\
	& \geq
	& \limsup_{n\to\infty}\frac{1}{|\cF_n|}\sum_{g\in\cF_nh^{-1}}\operatorname{Diam}_m(g(B(hx,\eta_\delta))) \\\nonumber
	&= & \limsup_{n\to\infty} \frac{1}{|\cF_n
	h^{-1}|} \sum_{g\in\cF_nh^{-1}} \operatorname{Diam}_m(g(B(hx,\eta_\delta))) \label{eq:criticalstepfolner} \\\nonumber
	& = & \limsup_{n\to\infty} \frac{1}{|\cF_n
	h^{-1}|} \sum_{g\in\cF_nh^{-1}} \operatorname{Diam}_m(g(B(hx,\eta_\delta))) \
	> \ \varepsilon \,.  \end{eqnarray} In order to justify the step in
	(\ref{eq:criticalstepfolner}) we use the \Folner property of $\cF$. A
	straightforward calculation yields \begin{align*}
	&\limsup_{n\to\infty} \frac{1}{|\cF_n
	h^{-1}|} \sum_{g\in\cF_nh^{-1}} \operatorname{Diam}_m(g(B(hx,\eta_\delta)))
	- \limsup_{n\to\infty} \frac{1}{|\cF_n|} \sum_{g\in\cF_n} \operatorname{Diam}_m(g(B(hx,\eta_\delta))) \\ \leqslant
	& \limsup_{n\to\infty} \frac{1}{|\cF_n
	h^{-1}|} \sum_{g\in\cF \triangle \cF_nh^{-1}} \operatorname{Diam}_m(g(B(hx,\eta_\delta))) \\ \leqslant
	& \diam_m(X) \cdot \limsup_{n\to\infty} \frac{|\cF \triangle \cF_nh^{-1}|}{|\cF_n
	h^{-1}|} \  = \ 0 \,.  \end{align*} So $x \notin \fs_\varepsilon$.  By
	contraposition we obtain that $x \in \fs_\varepsilon$ implies
	$hx \in \fs_\varepsilon$ --- \ie the desired invariance.  Replacing $x$
	by $h(x)$ and $h$ by $h^{-1}$, we also obtain the reverse inclusion and
	the invariance of $(\de_\varepsilon)^c$.
\end{proof}

\begin{prop}[Auslander-Yorke Dichotomy for Diam-Mean Equicontinuity]
	\label{prop2}
	\label{prop:auslanderyorkediammean}
	Let $\xg$ be a transitive \tds.
	Then $\xg$ is either almost diam-mean $m$-equicontinuous
	--- \ie $\de$ is a dense $G_\delta$-set ---
	or $\xg$ is diam-mean $m$-sensitive --- \ie $\de = \emptyset$.

	If in addition, $\xg$ is minimal, then either $\xg$ is diam-mean $m$-equicontinuous
	--- \ie $\de = X$ ---
	or $\xg$ is diam-mean $m$-sensitive
	--- \ie $\de=\emptyset$.
\end{prop}
The proof is analogous to the proof of
	Proposition \ref{prop:auslanderyorkefreqstab}. We include it for the
	convenience of the reader.  
	\begin{proof} Let $\xg$ be transitive with a transitive point $x\in X$.

		Assume on the one hand that $x \in \de$.  Then $\de$ is a dense
		and by Proposition \ref{prop:gdeltaandinvariancediammean} $\de$
		is a $G_\delta$-set.  Thus $\de$ is dense $G_\delta$ and hence
		$\xg$ is almost diam-mean $m$-equicontinuous.

		Assume on the other hand that $x \notin \de$.
		Then there is $\varepsilon > 0$ such that
		$x \in (\de_\varepsilon)^c$.
		By the invariance of $(\de_\varepsilon)^c$, we have $Gx \subseteq (\de_\varepsilon)^c$.
		Finally, by the closedness, we have $X = \overline{Gx} \subseteq (\de_\varepsilon)^c$.
		Therefore $\xg$ is diam-mean $m$-sensitive.
		It is in particular clear, that the cases are disjoint.

		Let $\xg$ now be minimal.
		Then any $x \in X$ is a transitivity point.
		Assume that there is $\varepsilon > 0$ and $y \in X$ such that $y \in (\de_\varepsilon)^c$.
		Then again --- as $y$ has dense orbit --- so $X = (\de_\varepsilon)^c$.
		So $\de = \emptyset$.
	\end{proof}

\section{Multivariate diam mean equicontinuity and multivariate regularity}\label{sec:4}
The following result relates diam mean $m$-equicontinuity and $m$-regularity.
In Section~\ref{WeaklyMeanSensitiveTuples}, we also prove the converse under the
assumption of a local Bronstein condition. 
In the abelian case, the same results
have been established in \cite{Lino2025}. 
\begin{thm}\label{thm:regular m to one=>m-equi}
	Let $(X,G)$ be a minimal \tds. If $\pi_{\eq}$ is regular $m$-to-one, then $(X,G)$
	is diam-mean $(m+1)$-equicontinuous.
\end{thm}
\begin{proof}
By the above, we only need to prove that any $x\in X$ is a diam-mean
$m$-equicontinuous point. For convenience, we assume without loss of generality
that $\diam(X)=1$.
	
For any $\varepsilon>0$, since $\pi_{\eq}$ is $m$-regular, we can choose a
compact subset $K$ of $X_{\eq}$ with $\mu_{\eq}(K)>1-\varepsilon/3$ and
$|\pi_{\eq}^{-1}(y)|=m$ for any $y\in K$. Since $\pi_{\eq}^{-1}$ is upper
semi-continuous, there exists $\delta>0$ and $z_1,\ldots z_l\in K$ such that
$K\subset \cup_{i=1}^lB(z_i,\delta)$ and $\pi_{\eq}^{-1}(B(z_i,2\delta))\subset
B(\pi_{\eq}^{-1}({z_i}),\varepsilon/3)$. Now we choose $\eta>0$ such that
$\mathrm{diam}(\pi_{\eq}(B(x,\eta)))<\delta$ holds for any $x\in X$. Notice that
for any $x\in X$, if $g\in G$ such that $g\pi_{\eq}(x)=\pi_{\eq}(gx)\in K$,
there exists $j\in\{1,2,\ldots,l\}$ such that $\pi_{\eq}(gx)\in
B(z_i,\delta)$. Since
$\mathrm{diam}(g\pi_{\eq}(B(x,\eta)))=\mathrm{diam}(\pi_{\eq}(B(x,\eta)))<\delta$,
it follows that $\pi_{\eq}(gB(x,\eta))\subset B(z_i,2\delta),$. This further
implies that
\[
gB(x,\eta)\ \subset\ \pi_{\eq}^{-1}( B(z_i,2\delta)) \ \subset
\ B(\pi_{\eq}^{-1}({z_i}),\varepsilon/3) \ .
\]
As $|\pi^{-1}(z_i)|=m$, it follows that $\diam_{m+1}(gB(x,\eta))<2\varepsilon/3$. Thus,
we deduce that
\begin{eqnarray} \nonumber
	\lefteqn{\limsup_{n\to\infty}\frac{1}{|\cF_n|}\sum_{g\in
            \cF_n}\operatorname{Diam}_{m+1}(gB(x,\eta))}\\ &= & \label{eq:14}
        \limsup_{n\to\infty}\frac{1}{|\cF_n|}\left(\sum_{g\in \cF_n,\pi(gx)\in
          K}\operatorname{Diam}_{m+1}(gB(x,\eta))+\sum_{g\in \cF_n,\pi(gx)\notin
          K}\operatorname{Diam}_{m+1}(gB(x,\eta))\right)\\ &= &
        \frac{2\varepsilon}{3}\cdot \nonumber
        \limsup_{n\to\infty}\frac{1}{|\cF_n|}\left|\left\{g\in \cF_n:g\pi(x)\in
        K\right\}\right| \\ & & + \ \diam(X)\cdot
        \limsup_{n\to\infty}\frac{1}{|\cF_n|}\left|\left\{g\in \cF_n:g\pi(x)\notin
        K\right\}\right| \ .\nonumber
\end{eqnarray}
Since $(X_{\eq},G)$ is uniquely ergodic, by the Birkhoff Ergodic Theorem (by
Lindenstrauss \cite{L01}), for any tempered \Folner sequence
$\F=\{\cF_n\}_{n\in\N}$, \eqref{eq:14} implies that
\[\limsup_{n\to\infty}\frac{1}{|\cF_n|}\sum_{g\in \cF_n}\operatorname{Diam}_{m+1}(gB(x,\eta))
\le 2\varepsilon\cdot\mu_{\eq}(K)+(1-\mu_{\eq}(K))<\varepsilon \ . \] As any \Folner
sequence has a tempered \Folner subsequence \cite{L01}, this completes the
proof.
	\end{proof}

We now turn to the notion of diam mean sensitivity.
	
\begin{prop}\label{prop:transitive ponit}
Let $(X,T)$ be a \tds with a transitive point $x$. If $x$ is diam-mean
$m$-sensitive then $(X,T)$ is diam-mean $m$-sensitive.
\end{prop}
\begin{proof}
Since $x$ is diam-mean $m$-sensitive, there exists $\varepsilon>0$ such that for any
open neighborhood $U$ of $x$,
\[
\sup_{\cF}\limsup_{n\to\infty}\frac{1}{|\cF_n|}\sum_{g\in
  \cF_n}\operatorname{Diam}_m(g(U)) \ \geq \ \varepsilon \ .
\]
Fix a nonempty open subset $V$ of $X$. As $x$ is transitive, there exists $h\in
G$ such that $hx\in V$. By continuity of $h$, we can find an open neighborhood
$U$ of $x$ such that $hU\subset V$. Thus, we have
\begin{align*}
	\sup_{\cF} \limsup_{n\to\infty}\frac{1}{|\cF_n|}\sum_{g\in
          \cF_n}\operatorname{Diam}_m(g(V)) \ \ge
        \ \sup_{\cF}\limsup_{n\to\infty}\frac{1}{|\cF_nh|}\sum_{g\in
          \cF_nh}\operatorname{Diam}_m(g(U)) \ > \ \varepsilon.
\end{align*}
\end{proof}



\section{Weakly mean sensitive tuples}\label{WeaklyMeanSensitiveTuples}

\subsection{Multivariate diam mean sensitivity and weakly mean sensitive tuples}
\label{MDiamMeanSensitivityAndTuples}

Following ideas in \cite[Theorem 4.4]{LiYeYu2022Equiinthemean}, we first prove
the following result.
\begin{thm}\label{lem:wealy tuple vs diammean sensitive}
Let $(X,G)$ be a \tds and $m\in\N$. If $\xg$ has an essential weakly
mean-sensitive $m$-tuple, then $(X,G)$ is diam-mean $m$-sensitive. If we further
assume that $\xg$ is transitive, then the converse direction also holds.
\end{thm}
\begin{proof}
First, we assume that $(x_1,\ldots,x_{m})$ is an essential weakly mean-sensitive
$m$-tuple. Let $\delta:=\frac{1}{2}\min_{1\le i< j\le m}d(x_i,x_j)$. We choose a
neighborhood $U_k$ of $x_k$ for each $k\in\{1,2\cdots,m\}$ such
that
\begin{equation}\label{eq1}
  \min_{1\le i< j\le m}d(U_i,U_j)>\delta.
\end{equation}
Then there exists $\eta>0$ such that for any open subset $U$ of
$X$, $$\uBD(\{g\in G: (gU)\cap U_k\neq\emptyset\text{ for }1\le k\le
m\})>\eta.$$ Let
\[
\mathcal{N}:=\{g\in G: (gU)\cap U_k\neq\emptyset\text{ for }1\le k\le m\}.
\]
Then it is easy to see that for any $g\in\mathcal{N}$, $\diam_m(g(U))>\delta.$
Thus,
\[
\sup_{\cF}\limsup_{n\to\infty}\frac{1}{|\cF_n|}\sum_{g\in
  \cF_n}\operatorname{Diam}_m(g(U))\ >\ \delta\eta\ .
\]
Therefore, $\xg$ is diam-mean $m$-sensitive with constant
$\varepsilon=\delta\eta$. \medskip

In order to prove the converse, we suppose that $\xg$ is a transitive diam-mean
$m$-sensitive \tds. Then one can show by elementary arguments that there exists
$\varepsilon>0$ such that for any nonempty open set $U$ of $X$, there exists a
subset $F$ of $G$ with $\uBD(F)\geq\varepsilon$ such that for any $g\in
F$,
\[\diam_m(g(U))\ \geq \ \varepsilon\ .\]
This implies that for any $k\in F$ there exist $x_1^k,\ldots,x_{m}^k\in U$ such
that
\begin{equation}\label{eq:123}
	\min_{1\le i< j\le m}d(kx^k_i,kx^k_j)\ge \varepsilon.
\end{equation}
	
Put
\[
X^{m}_\varepsilon\ := \ \{(x_1,\ldots,x_{m})\in X^{m}:\min_{1\le i< j\le
  m}d(x_i,x_j)\ge \varepsilon\}\ .
\]
Then $X^{m}_\varepsilon$ is a compact subset of $X^{m}$. Fix a transitive point
$x\in X$ and set $W_n:=B(x,1/n)$ for each $n\in\N$. By \eqref{eq:123}, for each
$n\in\N$,
\[
\uBD(\{g\in G:(gW_n)^{m}\cap X_\varepsilon^{m}\neq
\emptyset\})\ \geq\ \varepsilon\ .
\]
Due to the compactness of $X^{m}_\varepsilon$, there exists a finite family of open
sets $\{A_i^1\}_{i=1}^{N_1}$ with $\diam(A_i^1)\le 1$ such that
\[
X^{m}_\varepsilon\ \subset\ \cup_{i=1}^{N_1}A_i^1\ .
\]
By the pigeonhole principle, we may assume that for any $n\in\N$, there exists
$p_n\in\{1,2,\ldots,N_1\}$ such that
\[
\uBD(\{g\in G:(gW_n)^{m}\cap A_{p_n}^1\cap X_\varepsilon^{m}\neq
\emptyset\})\geq\varepsilon/N_1.
\]
Since the sets $W_n$ are decreasing and $N_1<\infty$, we may choose the same $p_n$ for all $n\in \N$. We can thus assume
without loss of generality that $p_n=1$ for all $n\in\N$.
		
Repeating the produce above, for each $l\ge1$, we may assume that there exists
an open subset $A^l$ with $\diam(A^l)<1/l$ such that $A^l\subset A^{l-1}$ and
for all $n\in\N$,
\begin{equation}\label{eq:20.44}
	\uBD\left(\left\{g\in G:(gW_n)^{m}\cap A^l\cap
        X_\varepsilon^{m}\neq \emptyset\right\}\right)\ \geq
        \ \frac{\varepsilon}{N_1N_2\ldots N_l} \ .
\end{equation}
	
As $\lim_{l\to\infty}\diam(A^l)=0$, it is clear that $(\cap_{l=1}^\infty
A^l)\cap X_\varepsilon^{m}$ is a singleton, denoted by $(x_1,\ldots,x_{m})$. Now we
prove this point is a weakly mean-sensitive tuple. Indeed, for any family of
open neighborhoods $U_i$ of $x_i$, $i=1,2,\ldots,m$, there exists $l>0$ such
that $A^l\cap X_\varepsilon^{m}\subset U_1\times\ldots\times U_{m}$.  By
\eqref{eq:20.44}, for all $n\in\N$,
\[
\uBD\left(\left\{g\in G:(gW_n)^{m}\cap (U_1\times\ldots\times
U_{m})\neq \emptyset\right\}\right)\ \geq \ \frac{\varepsilon}{N_1N_2\ldots N_l}\ .
\]
For any open subset $U$ of $X$, as $x\in X$ is a transitive point, there exists
$h\in G$ such that $hx\in U$. By the continuity of $h$, we can choose $n\in\N$
large enough such that $hW_n\subset U$. Then we have
\begin{align*}
	\uBD\left(\left\{g\in G:(gU)\cap U_i\neq \emptyset,\text{ for
        }i=1,2,\ldots,m\right\}\right) \ \geq \ \frac{\varepsilon}{N_1N_2\ldots
          N_l}\ .
\end{align*}
This completes the proof.
\end{proof}

In order to prove a converse to Theorem~\ref{thm:regular m to one=>m-equi}, we
need an additional condition that is required in order to apply results from
\cite{LiuXuZhang2025}. We say that a \tds $\xg$ satisfies the {\em local
  Bronstein condition} if for any minimal point $(x,x')\in X\times X$ such that
$\pi_{\eq}(x)=\pi_{\eq}(x')$ there exists a sequence $((x_n,x_n'))_{n\in\N}$ of
minimal points in $X\times X$ and a sequence $(g_n)_{n\in\N}$ in $G$ such that
$\lim_{n\to\infty} (x_n,x_n')=(x,x')$ and $\lim_{n\to\infty} d(g_nx,g_nx')=0$.
Note that this is always satisfied if $(X,G)$ is a minimal action of an abelian
group (Lemma~\ref{lem:nilpotentimpliesbronstein}) The following statement now
summarises our findings combined with the respective results in
\cite{LiuXuZhang2025}.

\begin{thm}\label{thm:m to 1<=> tuple}
Suppose that $(X,G)$ is a minimal \tds that satisfies the local Bronstein
condition. Let $m\in \N$. Then the following
statements are equivalent:
\begin{enumerate}[(i)]
	\item  $\pi_{\eq}$ is regular $m$-to-one;
	\item there exists an essential weakly $m$-sensitive tuple and no
          essential weakly $(m+1)$-sensitive tuple;
	\item for  every $y\in X_{\eq}$, $\pi_{\eq}^{-1}(y)$ contains an
          essential weakly $m$-sensitive tuple and no essential weakly
          $(m+1)$-sensitive tuples;
          \item $(X,G)$ is diam-mean $(m+1)$-equicontinuous but not
                  diam-mean $m$-equicontinuous;
	
		\item $(X,G)$ is not diam-mean $(m+1)$-sensitive, but diam-mean
                  $m$-sensitive.
\end{enumerate}
\end{thm}
\begin{proof}
(i) $\Rightarrow$ (iii). The first statement has been proven in \cite[Corollary
    5.3]{LiuXuZhang2025} . Now, assume for a contradiction that there exists an
  essential weakly $(m+1)$-sensitive tuple. By Theorem~\ref{lem:wealy tuple vs
    diammean sensitive}, we know that $(X,G)$ is diam-mean
  $(m+1)$-sensitive. Hence Proposition \ref{prop:auslanderyorkediammean} implies
  that $(X,G)$ is not diam-mean $(m+1)$-equicontinuous. However, this
  contradicts Theorem \ref{thm:regular m to one=>m-equi}.
 
(iii) $\Rightarrow$ (ii) is trivial.

(ii) $\Rightarrow$ (i).  By Corollary 5.3 in \cite{LiuXuZhang2025}, as $(X,G)$
  has no essential weakly $(m+1)$-sensitive tuple, we know that $\pi_{\eq}$ is
  regular $m'$-to-one for some $0<m'\le m$. As $(X,G)$ has an essential weakly
  $m$-sensitive tuple, the fact that (i) $\Rightarrow$ (iii) yields $m'\geq
  m$. Therefore (i) holds.
  
(i) $\Rightarrow$ (iv) has been proven in Theorem \ref{thm:regular m to
  one=>m-equi}.

(iv) $\Leftrightarrow$ (v) holds by Proposition \ref{prop:auslanderyorkediammean}.

(v) $\Rightarrow$ (ii) finally follows from Theorem \ref{lem:wealy tuple vs
    diammean sensitive}.
   \end{proof}

\subsection{Strong $m$-spreading and mean sensitive tuples}
\label{Strong_m_spreading}

For the sake of completeness, we discuss in this section how strong
$m$-spreading may be characterised via suitable finite collections of weakly
mean sensitive tuples. However, due to the fact that the spreading must occur
along sets of full density, it turns out that this cannot be achieved by looking
at single tuples. Instead, finite collections of suitable tuples have to be
considered.

We call a \tds $\xg$ {\em strongly $(\varepsilon,m)$-spreading} for some
$\varepsilon>0$ and $m\in\N$ if
\[
\lBD\left(\left\{ g\in G\mid \diam_m(gU)
>\varepsilon\right\}\right) \ = \ 1 \
\]
holds for any open subset $U \subseteq X$. Note that $\xg$ is {\em strongly
  $m$-spreading} if it is strongly $(\varepsilon,m)$-spreading for some $\varepsilon
>0$. Recall that we defined
\[
D_m(x) \ = \ \min_{1\leq i < j\leq m} d(x_i,x_j) \ 
\]
and 
\[
X^m_{\varepsilon} \ = \ \left\{ x\in X^m\mid D_m(x) \geq \varepsilon\right\} .
\]
Now strong $m$-spreading can be characterised in the following way. 
\begin{lem}
  If $\xg$ is strongly $(\varepsilon,m)$-spreading, then for all $\delta>0$ there
  exists $N=N(\delta)\in\N$ and $x^{(1)},\ldots, x^{(N)}\in X^m_{\varepsilon}$ such
  that
  \begin{equation}\label{e.strong_spreading_tuples_1}
 \lBD\left(\left\{ g\in G \ \left|\ \exists k\in\{1,\ldots, N\} :
 gU\cap B\left(x^{(k)}_i,\delta\right)\neq \emptyset \textrm{ for } i=1,\ldots,
 m\right\}\right)\right. \ = \ 1 \ .
  \end{equation}
holds for all nonempty open sets $U\subseteq X$.
\end{lem}
\begin{proof}
  Note that if a set $V\subseteq X$ satisfies $\diam_m(V) >\varepsilon$, then $V^m$
  intersects $X^m_{\varepsilon}$. Hence, as $\xg$ is strongly
  $(\varepsilon,m)$-spreading, we have that
  \begin{equation} \label{e.strong_spreading_hyperspace}
  \lBD \left(\left\{ g\in G \mid g(U)^m\cap X^{m,\varepsilon}\neq
  \emptyset \right\}\right) \ = \ 1 \ .
  \end{equation}

  Now fix $\delta>0$. Due to compactness, there exists $N\in\N$ and
  $x^{(1)},\ldots, x^{(N)}\in X^m_{\varepsilon}$ such that the product
  neigbourhoods
  \[
  U\left(x^{(k)},\delta\right) \ = \ \bigotimes_{k=1}^m B\left(x^{(k)}_i,\delta\right)
  \]
  with $k=1,\ldots,N$ cover $X^m_{\varepsilon}$. As a consequence,
  (\ref{e.strong_spreading_hyperspace}) implies
  (\ref{e.strong_spreading_tuples_1}).
\end{proof}
Conversely, we have
\begin{lem}
  Suppose $\xg$ be a \tds and $\varepsilon>0$ and $\delta\in (0,\varepsilon/2)$. If
  there exists $N=N(\delta)\in\N$ and $x^{(1)},\ldots, x^{(N)}\in
  X^{m}_{\varepsilon}$ such that
  \[
   \lBD\left(\left\{ g\in G \ \left| \exists k\in\{1,\ldots, N\} :
   gU\cap B\left(x^{(k)}_i,\delta\right)\neq \emptyset \textrm{ for }
   i=1,\ldots, m\right\}\right)\right. \ = \ 1 \ ,
    \]
    holds for all nonempty open sets $U\subseteq X$, then $\xg$ is strongly
    $(\varepsilon-2\delta,m)$-spreading.
\end{lem}
\begin{proof}
  The statement follows directly from the observation that $x\in X^m_{\varepsilon}$
  and $V\cap B(x_i,\delta)\neq \emptyset$ for $i=1,\ldots,m$ implies $\diam_m(V)
  > \varepsilon-2\delta$.
\end{proof}

Together, the previous two lemmas yield
\begin{prop}\label{p.strong_spreading_via_tuples}
  Let $\xg$ be a \tds, $m\in\N$ and $\varepsilon_0>0$. Then the following are
  equivalent.
  \begin{itemize}
  \item[(i)] $\xg$ is strongly $(\varepsilon,m)$-spreading for all
    $\varepsilon<\varepsilon_0$;
  \item[(ii)] for all $\varepsilon < \varepsilon_0$ and $\delta>0$ there exist
    $N=N(\varepsilon,\delta)$ and $x^{(1)},\ldots, x^{(N)}\in X^m_{\varepsilon}$ such
    that
  \[
  \lBD\left(\left\{ g\in G \ \left| \exists k\in\{1,\ldots, N\} :
  gU\cap B\left(x^{(k)}_i,\delta\right)\neq \emptyset \textrm{ for } i=1,\ldots,
  m\right\}\right)\right. \ = \ 1 \ ,
    \]
    holds for all nonempty open sets $U\subseteq X$
  \end{itemize}
\end{prop}


\bibliographystyle{acm} 
\def\cprime{$'$} \def\cprime{$'$}

\end{document}